\documentclass[12pt,draftcls,onecolumn]{IEEEtran}

\newif\iflong
\longfalse

\usepackage[utf8]{inputenc}

\usepackage{amsmath}
\usepackage{amssymb}
\usepackage{amsthm}
\usepackage{bbm}
\usepackage{nicefrac}
\usepackage[fixamsmath]{mathtools}
\usepackage[ruled,linesnumbered]{algorithm2e}
\usepackage{enumitem}
\setenumerate[1]{label=(\roman*)}
\setenumerate[2]{label=(\alph*)}

\usepackage{todonotes}

\usepackage{accents}

\usepackage{subcaption}
\usepackage{graphicx}
\usepackage{hyperref}
\usepackage{cleveref}
\usepackage{xcolor}

\newcommand{\E}{\mathbb{E}}
\newcommand{\N}{\mathbb{N}}

\newcommand{\R}{\mathbb{R}}

\newcommand{\cG}{\mathcal{G}}

\newcommand{\cO}{\mathcal{O}}
\newcommand{\cP}{\mathcal{P}}
\newcommand{\cS}{\mathcal{S}}

\newcommand{\cX}{\mathcal{X}}
\newcommand{\cY}{\mathcal{Y}}

\renewcommand{\path}[2]{{#1\xrightarrow{\mathcal{P}}#2}}

\newcommand{\Ew}[1]{\mathbb{E}\left[#1\right]}
\renewcommand{\Pr}[1]{\mathbb P\left(#1\right)}

\newcommand{\norm}[1]{\lVert #1 \lVert}

\newcommand{\abs}[1]{\vert #1 \vert}
\newcommand{\Ind}[1]{\mathbbm{1}_{ #1 }}

\newtheorem{theorem}{Theorem}
\newtheorem{lemma}{Lemma}

\newtheorem{definition}{Definition}

\newtheoremstyle{asmstyle}
{\topsep} 
{\topsep} 
{} 
{} 
{\bfseries} 
{~(A\theassumption)} 
{.5em} 
{} 
\theoremstyle{asmstyle}\newtheorem{assumption}{Assumption}
\crefname{assumption}{}{}
\creflabelformat{assumption}{#2(A#1)#3}

\usepackage{fancyhdr}
\fancypagestyle{firstpage}{%
\lhead{This work has been submitted to the IEEE for possible publication.  Copyright may be transferred without notice, after which this version may no longer be accessible.}
}
\title{Practical sufficient conditions for convergence of distributed optimisation algorithms over communication networks with interference}
\author{Adrian Redder$^{\S,\dag}$~\texttt{aredder@mail.upb.de} \thanks{$^\S$Supported by the German Research Foundation (DFG) - 315248657.} \thanks{$^\dag$ Computer Networks Group - Dept. of Computer Science Paderborn University.} \and \\
Arunselvan Ramaswamy$^\ddag$~\texttt{arunr@mail.upb.de} \thanks{$\ddag$ Heinz Nixdorf Institute and the Dept. of Computer Science
Paderborn University} \and \\
	Holger Karl$^\dag$~\texttt{hkarl@ieee.org} 
}
\date{\today}

\begin{document}
\maketitle
\thispagestyle{firstpage}

\begin{abstract}
	Information exchange over networks can be affected by various forms of delay. This causes challenges for using the network by a multi-agent system to solve a distributed optimisation problem. 
	Distributed optimisation schemes, however, typically do not assume network models that are representative for real-world communication networks, since communication links are most of the time abstracted as lossless.
	Our objective is therefore to formulate a representative network model and provide practically verifiable network conditions that ensure convergence of distributed algorithms in the presence of interference and possibly unbounded delay.
	Our network is modelled by a sequence of directed-graphs, where to each network link we associate a process for the instantaneous signal-to-interference-plus-noise ratio. We then formulate practical conditions that can be verified locally and show that the age of information (AoI) associated with data communicated over the network is in $\cO(\sqrt{n})$. Under these conditions we show that a penalty-based gradient descent algorithm can be used to solve a rich class of stochastic, constrained, distributed optimisation problems.
	The strength of our result lies in the bridge between practical verifiable network conditions and an abstract optimisation theory. We illustrate numerically that our algorithm converges in an extreme scenario where the average AoI diverges.
\end{abstract}

\section{Introduction}
\label{sec: intro}

Distributed constrained stochastic optimisation problems lie at the heart of many system-level problems such as multi-agent learning or wireless network control. 
Here, we consider that $D$ agents should choose values for local variables $x_i \in \R^{d_i}$ to minimise a real-valued function $\Ew{f(x_1, \ldots x_D,\xi)}$ with respect to a random variable $\xi$ that takes on values in a compact set $\cS$. The optimisation variable $x = (x_1, \ldots, x_D) \in \R^d = \R^{\sum_{i=1}^{D} d_i} $ is the concatenation of the local optimisation variables $x_i$ associated with the local agents $i \in \{1,\ldots,D\}$. Moreover, $x$ should take on values in a constrained set $\cX$.
Hence, the problem is to find
\begin{equation}\label{eq: intro_problem}
\begin{split}
x^* = (x_1^*, \ldots, x_D^*) &= \underset{x \in \R^d}{\text{argmin }} \E_\xi \left[ f(x,\xi)\right],  \\
\text{subject to}& \qquad  x \in \cX,
\end{split} 
\end{equation}
\noindent where each agent $i$ has to select $x_i^*$ locally. For now, we do not put any assumptions on $\cX$, but in the following sections we will put assumptions on a penalty function $P$ that characterises $\cX$.  We will assume that both $f$ and $\cX$ are known to all agents but our discussion can be easily extended to scenarios with local objectives and constraints. In general, an agent cannot decide whether a change of its local variable leads to a violation of the constraints (it could only decide that if it had access to the exact local variables from the other agents, which we do not assume). 

To solve this problem, agents exchange data over a communication network that might loose or delay data or whose topology may vary over time. The optimisation algorithm of the agents should be robust to converge to a solution of the problem in the presence of these factors. We consider that the agents iteratively update their local variable using a distributed stochastic gradient descent (SGD) algorithm starting from an initial guess $x_i^0$ of $x_i^*$. The agents thereby generate a sequence  $\{x_i^n\}_{n\ge0}$ of candidate solutions for $x_i^*$ at some time $n$, where we use a time index $n$ according to a global clock $n$. This clock is of course unavailable to each individual agent and only used for notational convenience (e.g., it is not necessary that \emph{all} agents update their variable at time $n$; typically, each agent will update its variable asynchronously at the tic of a local clock that occurs in some time interval between two ticks of the global clock.)

To update their local variable, each agent can only calculate gradients $\nabla_{x_i} f(\cdotp, \xi)$ for samples of the random variable $\xi$.
To calculate the gradients and update $x^n_i$, agent~$i$ would like to have access to $\{x_j^n\}_{j \neq i}$, the current local variables of the other agents, yet these variables are only available with unknown delays $\tau_{ji}(n)$. 
Hence, only $\{x_j^{n-\tau_{ji}(n)}\}_{j \neq i}$ might be available at agent~$i$ at time $n$, where $x_j^{n-\tau_{ji}(n)}$ is the newest update of $x_j$ available at agent~$i$ at time $n$ and $\tau_{ji}(t)$ its corresponding age. \emph{We refer to $\tau_{ji}(t)$'s as the Age of Information (AoI) variables.} This AoI is driven by the lossy, delaying network. 
Note that as information might flow via multiple hops in the network, a possibly time-varying network topology affects the  AoI between any two nodes $i$ and $j$.\footnote{It is simple to ensure, using sequence numbers, that an agent $i$  only uses the newest data from another agent $j$; if older data should arrive after newer data has already been received, it is simply discarded.} 

The main challenges of the above setup is that information between any pair of agents might experience unbounded delays: Mobile agents or network scheduling algorithms induce a time-varying set of network topologies, where agents might be unable to exchange data for extended periods of time. Additionally, interference from other transmissions in the network might result into repeated packet losses. 

Another challenge is that agents should be allowed to act completely asynchronously, since synchronisation of all agents in a large network is practically unrealistic. In this work, agents update their local variable at the tick of a local clock
and communication with other agents is driven by these updates as well as by the received data from the other agents. Communication is therefore asynchronous and event-driven: New data arrives at an agent asynchronously; At a tick of its clock an agent computes a new update for its local variable using available information; The new update is scheduled to be transmitted to other agents. The joint sequence of local variables should converge to a candidate solution for $x^*$ in this asynchronous setting.


In the described networking scenario, it is straightforward to run a distributed SGD algorithm on problem \ref{eq: intro_problem}. At each agent $j$ the optimisation variables $\{x_j^n\}_{j \neq i}$, which would be used in a centralised SGD algorithm, are replaced by the delayed optimisation variables $\{x_j^{n-\tau_{ji}(n)}\}_{j \neq i}$. However, it is unclear whether this approach can indeed converge to a solution of problem \ref{eq: intro_problem} and which network conditions need to be imposed to guarantee convergence. 
The main question of this paper is therefore whether we can characterise practical network conditions under which we can use a distributed implementation of centralised SGD to solve problem \ref{eq: intro_problem}, where the central variables are simply replaced by their delayed counterparts. Moreover, can these conditions be checked from the agents themselves, or only from a (hypothetical) outside observer with global network information.
Our paper provides sufficient conditions on the network under which convergence is ensured and that can be checked locally by the agents themselves -- we call such conditions  ``practically verifiable''.

Most optimisation schemes for time-varying networks consider very restrictive network models. For example, Wang et al.\ \cite{wang2019distributed} consider doubly-stochastic network graphs where communication occurs periodically.  In \cite{hendrikx2019asynchronous}, communication happens via a fixed, regular communication graph with fixed delay; in \cite{Lei2018-ag}, the network is modelled as an i.i.d.\ sequence of directed graphs where the mean graph is connected and communication is guaranteed with additive communication noise if a network edge is present. Yu et al.\ \cite{Yu2020-gq} consider doubly-stochastic network graphs where the union graph of the time-varying network topology over a fixed period has to be strongly connected;  additionally, independent additive communication noise is considered. Ref.\ \cite{Scutari2019-sa} considers a slightly weaker form of doubly-stochastic network graphs but also that the union graph of the time-varying network topology over a fixed period has to be strongly connected.

Importantly, in all of that work (expect \cite{hendrikx2019asynchronous}) the assumed network matrices at a specific time step is representative for the underlying network graph at that time step and a network edge represents a lossless communication link, i.e. communication is guaranteed if network edge is present though most works add additive communication to the communicated data. We believe that this is a restrictive assumption since in a practical scenario an external scheduling algorithm might assign a network channel to a user (e.g. a frequency band in an FDMA type protocol), but this may not guarantee successful data exchange over the whole assignment duration, e.g. due to interference from the adjacent frequency bands. 

The focus of most of the above related work is on the optimisation procedure and the assumed stochastic matrices are used in the algorithm of the agents to perform some form of averaging over agents. Our focus is to provide a less restrictive network model and practically verifiable network conditions. However, other papers such as those discussed above can provide stronger convergence results.
Our work contributes to the literature of network conditions that guarantee asymptotic convergence to the set of local minimisers of distributed stochastic constrained optimisation problems. Most importantly, the network conditions allow time-varying network topologies, unbounded communication delays, interfering network transmissions, asynchronous local updates and event-driven transmissions. \emph{To the best of our knowledge, ours is the first work that guarantees asymptotic convergence under such only mildly restrictive conditions,  connecting abstract optimisation theory to a practical network model.}

For the SGD algorithm, we merely consider a simple penalty-based technique to deal with the constraints in problem \ref{eq: intro_problem}. Classically, penalty-based optimisation techniques solve unconstrained versions of \eqref{eq: intro_problem} sequentially. Here, a penalty parameter is decreased once the unconstrained optimisation problem has been solved \cite[Ch.\ 23]{chong2013introduction}. The main issue of this approach is the choice of this penalty parameter. The choice is typically problem-dependent and may significantly impact the quality of optimisation. Additionally, a problem may require a choice such that the problem becomes ill-conditioned. While other techniques (e.g., Lagrangian multipliers \cite{venter2010reviewOpt}) might have advantages in our setting, for this paper we stick to a penalty-based approach and will revisit alternatives in future work.

In addition to problems of the form \eqref{eq: intro_problem}, many results in the distributed optimisation literature are currently devoted to cumulative or average consensus-type problems, where parts of the optimisation objective are associated with local agents. Algorithmically, these problems can be solved by our framework at the cost of exchanging gradient information instead of optimisation variables \cite{arxiv}. 

\emph{Main contribution:}
We propose practically verifiable network conditions that ensures asymptotic convergence of a distributed SGD algorithm to a local solution of a stochastic, constrained optimisation problem in the presence of interference and unbounded information delay.
Our network model is represented by a sequence of directed graphs where we associate to each edge of the time-varying network topology a signal-to-interference-plus-noise ratio (SINR) model.
We present condition \cref{asm: network_specific} for this model that ensures that the age of information associated with our distributed algorithm is in $\cO(\sqrt{n})$. 
We merely require that the SINR of each network edge satisfies some threshold with some non-zero probability that may decay to zero asymptotically. Second, we require that for each network edge the dependency on the history of the individual edge transmissions decay at some exponential rate. 
\emph{Our work therefore connects a network model that is representative for communication in wireless networks with an abstract optimisation theory that can solve a rich class of distributed, stochastic, constrained optimisation problems.}

\section{Preliminaries}
In this section, we set up notation and recall background from graph theory.
\subsection{Notation}
\begin{itemize}
	\item  $I$  denotes an arbitrary index set.
	\item Discrete points in time are indicated by superscript letters $n$. We refer to a time slot $n$ as the time interval from time step $n-1$ to $n$.
	\item $\Ind{X}$ denotes the indicator function of a set $X$.
\end{itemize}
\subsection{Graph theory}
A graph is pair $G = (V,E)$ with a set of nodes $V$ and a set of edges $E$. A directed graph is a graph where the edges are ordered pairs, i.e.\ an edge connecting two nodes has an associated direction.
A directed graph $G =(V,E)$ is called strongly connected if for any two nodes in $V$ there exists a path connecting them.  
A path is a sequence of edges $\{(v_i, v_{i+1})\}_{i \in I}$ with $v_i \in V$.
Given a set of graphs $\cG \coloneqq \{(V_i, E_i) \mid  i \in I  \}$, we refer to the union graph of $\cG$ as the graph $(\cup_{i\in I} V_i, \cup_{i\in I} E_i)$.
\section{Problem setup}
\label{sec: setup}
We consider a $D$-agent system $V \coloneqq \{1, \ldots, D\}$ for which the goal is to solve the distributed optimisation problem \eqref{eq: intro_problem} in a cooperative manner. The objective function $F(x) \coloneqq \E_\xi \left[ f(x, \xi) \right]$ is the expected value of the deterministic function $f(x,\xi)$ with respect to a stochastic component represented by $\xi$. Typically, $\xi$ may be some randomness that arises due to environmental or network-related fluctuations.
We reformulate problem \eqref{eq: intro_problem} as an unconstrained optimisation problem using a penalty function $P: \R^d \to \R$, which is assumed to be continuous, satisfies $P(x) > 0$ for all $x \in \cX^c$ as well as $P(x) = 0$ for all $x \in \cX$. The reformulated problem then becomes
\begin{equation}\label{eq: penalty_problem}
\underset{x \in \R^d}{\min}~ b F(x) + P(x), 
\end{equation}
with the penalty parameter $b > 0$.\footnote{In the literature, $b$ is usually a prefactor of $P(x)$. The formulations are equivalent, but the formulation in \eqref{eq: penalty_problem} is more convenient for analysis.} 
We will analyse the asymptotic theoretical behaviour as the parameter $b$ is gradually \emph{decreased} simultaneously with the step size of a stochastic gradient descend (SGD) iteration. 
The unconstrained reformulation is the basis for our penalty-based SGD iteration \eqref{eq: real_iteration} to be defined in \Cref{sec: gradient_iteration}. The iteration will describe how the agents update their local variables. To update their local  variables, the agents  exchange their current optimisation variable values over a communication network. The following four subsections will describe the networking environment of the agents. After that we formulate our SGD iteration and the local algorithm used by each agent as well as the associated assumptions.

\subsection{Asynchronous local updates}
\label{subsec: async_local_updates}

We assume that each agent has a local clock. Whenever the local clock ticks the agent updates its local variable. To take asynchronicity into account, we formulate the agents' local variables w.r.t. a time index $n$ of a hypothetical global clock that runs faster than any of the local clocks. There are various ways to define a global clock with this property. One way is to consider sequences $T^i \coloneqq \{t^i_m\}_{m\in \N}$ for the increasing time steps that represent the unknown points in time where the local clock of agent~$i$ ticks. Then, define the global clock $n$ as the enumeration of $ \cup_{i=1}^D T^i$ by increasing order. Clearly, $n$ ticks whenever any of the local clock ticks. In this work, we consider a hypothetical global clock with constant inter-tick time $\Delta$, i.e.\ let $T_n$ be the times where the global clock ticks, then $\abs{T_n - T_{n+1}} = \Delta \forall n$. \emph{We emphasize that the global clock is only for notational convenience in the following proofs; neither the clock nor its values appear anywhere in the actual system!}

For each agent~$i$, the local updates generate a sequence $\{x_i^n\}_{n\ge 0}$ starting from an initial candidate $x^0_i$ for $x^*_i$. In order to distinguish the age of the updates, each agent~$i$ appends a local timestamp $\nu(n,i)$ to its local optimisation variable $x_i^n$. For example, if agent 1 has updated its variable $4$ times while the global clock has ticked $10$ times, then $\nu(10,1) = 4$. From now on, when we say an agent transmits an optimisation variable $x_i^n$, we always mean that the variable is sent together with its associated local timestamp.

Agents iteratively refine their local variables using the partial derivatives $\nabla_{x_j} f(\cdot, \xi)$ and $\nabla_{x_j} P(\cdot)$.  Agents do not know the distribution of $\xi$, but during any time slot $n$ an agent can observe an i.i.d.\ realisation $\xi^n$ of $\xi$. For simplicity, we assume that all agents are affected by the same realisation of the random variable $\xi$. In other words, when agent~$i$ and $j$ calculate their partial derivatives during some time slot $n$, they use the same realisation $\xi^n$ of $\xi$, i.e. $\nabla_{x_i} f(\cdotp, \xi^n)$ and $\nabla_{x_j} f(\cdotp, \xi^n)$. The extension to agent-specific realisations of $\xi$ is merely a technical reformulation that was already mentioned in \cite{arxiv}.
To evaluate the partial derivatives $\nabla_{x_i} f(\cdotp, \xi^n)$, agent~$i$ requires a locally available estimate of the current optimisation variable $x^n_j$ of agent~$j$ for all $j\not=i$. These information will be communicated using a network.

\subsection{Graph-based network model}
\label{sec: graph-based formulation}
Each agent has to update its local variable based on information communicated using a network that connects the agents. We represent this network by a \emph{sequence of directed graphs} 
\begin{equation}
\cG \coloneqq \{G_n\}_{n \ge 0} \coloneqq \{(V,E^n)\}_{n \ge 0}.
\end{equation}
Each agent is in one-to-one correspondence with one node in the graph. Two agents can communicate directly\footnote{Communication between a pair of agents is considered ``direct'' when there is no involvement from other agents.} \textit{if and only if} the corresponding nodes are connected by an edge. Thus, an edge represents the availability of a channel for communication.
Agents can also communicate unidirectionally, which is represented using directional edges. It is important to note that (i) at no point in time, communication between agents is guaranteed to succeed when two nodes are connected by an edge, since the associated channel may experience unbounded delay, noise, or destructive interference, etc., and (ii) we explicitly consider changing graph topologies, i.e.\ agents that can communicated at time $n$ may not have a channel available at time $n+1$. We let $E \coloneqq \cup_{n\ge 0} E^n$ denote the union of all edge sets, such that $(V,E)$ is the union graph of $\cG$. For now, we consider no explicit model for the evolution of the network topologies. However, we assume that the sequence of directed graphs is stochastically strongly connected. 
\begin{definition}
	\label{def: stoch_strong_connected}
	We say a time-varying directed graph $\cG \coloneqq \{G^n\}_{n \ge 0}$ is \textbf{stochastically strongly connected} if the union graph of $\cG$ is strongly connected and there exists some $\varepsilon >0$ such that for every element $G \in \cG$ we have $\Pr{G^n = G} > \varepsilon$ for all $n\ge 0$.
\end{definition}
The above definition may be weakened in two ways. First, we only require $\Pr{G^n = G} > \varepsilon$ for those graphs $G$ that are necessary to make the union graph of $\cG$ strongly connected. 
Second, we can assume that $\Pr{G^n = G} > \varepsilon$ only holds periodically with an arbitrary large period. However, we choose the above definition to simplify the presentation. \emph{It is important to note that we do not require guaranteed transmissions periodically.}

\subsection{Channel model}
\label{sec: network_model}

This section presents the channel models that we associate with the edges in the time-varying network model of the previous subsection. To arrive at an expressive model, we follow the guidance given in \cite{Iyer2009_rightchannelmodel}. The authors suggest to represent network channels with interference by an additive signal-to-interference-plus-noise-ratio (SINR) model with a suitable signal power model. 

Fix a time slot $n$ and consider the associated network graph $G^n = (V,E^n)$ as well as an edge $(i,j) \in E^n$. To a transmission from node $i$ to node $j$ during time slot $n$ we associate the instantaneous $\text{SINR}^{n}_{ij}$. We assume that $\text{SINR}^{n}_{ij}$ is constant over one time slot. It is important to notice that this does not require some weak form of synchronisation between the agents, since for every sampling path the global clock $n$ used in the analysis may always be adapted to satisfy this requirement. 

In the following we will put assumptions on the event that the $\text{SINR}^{n}_{ij}$ satisfies some threshold.  In \Cref{sec: practical+ext} we will discuss practical verifiability and come back to these assumptions and relate them to assumptions on the signal power distributions in the standard additive model for the instantaneous SINR. 

Given the $\text{SINR}^{n}_{ij}$ in \eqref{eq: SINR_model_general}, we assume a constant SINR-threshold $\beta$ such that whenever the event $ A^n_{ij} \coloneqq \{\text{SINR}^{n}_{ij} \ge \beta\}$ occurs a transmission from node $i$ to $j$ can be successfully received; hence, errors are due to inference and noise-based errors are negligible. The threshold depends on the modulation, coding and path characteristics of each node pair. It is simple to extend the following analysis to node-pair-specific thresholds, so assuming only one SINR threshold does not loose generality. 
The effective rate for $A^n_{ij}$ is then given by the Shannon bound
for a bandwidth of $B$:
\begin{equation}
\label{eq: bitrate}
R < B \log_2(1+ \beta).
\end{equation} 

\noindent For a simpler presentation, we will not consider additional communication latencies, i.e. latency due to coding, propagation, etc.  Taking those into account would be  possible without too much technical difficulties.

\subsection{Communication protocol (CP)}
\label{sec: CP}
We consider a set of rules governing the data exchange between two nodes $i,j \in V$ that are connected by an edge $(i,j) \in E$ of the network. Due to the communication network each agent has only delayed information of the other agents´ optimisation variables. We let each agent maintain a local belief vector $ \hat{X}_i \coloneqq (\hat{x}_{1i}, \ldots, \hat{x}_{Di})$ of the global optimisation vector, where $\hat{x}_{ji}$ is the newest agent~$j$ estimate available to agent~$i$, which is possibly old.
Whenever the local belief vector $\hat{X}_i$ is updated, agent~$i$ schedules the updated components in $\hat{X}_i$ for a transmission to each of its potential neighbours in the time-varying graph $\cG$. Note that there is no need for an actual queue for each potential neighbour, since the only transmission that would be waiting are those in the local belief vector $\hat{X}_i$. Each agent therefore only has to update an order for the components in $\hat{X}_i$ for every other agent in the graph. This order can be created ones another agent appears for the first time as a neighbour. The order then determines the transmission order of the components in $\hat{X}_i$ whenever the associated network edge is available. We assume that each agent uses a suitable flooding protocol to reduce the number of possibly redundant transmissions \cite{lim2001flooding}.

Recall that a network edge represents the availability of a channel for communication. This assignment may for example be due to an external scheduling process and/or due to geographical connectivity of the agents. Whenever a network edge $(i,j) \in E$ is present and there are components in $ \hat{X}_i$ waiting for a transmission to agent $j$, we assume that agent $i$ uses a suitable medium access control mechanism to access the network channel. We leave it as future work to design MAC protocols that are directly targeted towards our distributed optimisation scenario. For example, agent $j$ seems like the natural candidate to estimate $A^n_{ij}$ and then send this as a clear to sent to agent $i$.

Note that there are two sources of asynchronicity in our setup. First, the agents update their optimisation variable according to their local clocks. Second, each agent only schedules transmissions for communication with its neighbours after its local belief vector has changed. Therefore, the information delay also leads to asynchronicity of the agents.
It is clear that the time-varying network topology and the probability of a successful transmission over a network edge affect the AoI variables. Further, these are also affected by the CP. Here, we do not consider an adaptation of the CP using information of the network topology\footnote{In general it is advantageous to design a topology-sensitive CP, i.e. a CP that is adapted to the current network topology. Leveraging topological knowledge can greatly reduce delays and hasten the rate of convergence of distributed optimisation algorithms.}.

\subsection{Gradient descent iteration and \Cref{alg: 1}}
\label{sec: gradient_iteration}

Recall that we defined $\tau_{ji}(n) \in \{0, \ldots, n\}$ as the (stochastic) delays experienced by agent~$i$ in receiving the most recent agent~$j$ estimate, at time $n$. For example, consider that agent~$j$ transmits $x^{t-3}_j$ to agent~$i$ and the packet has a delay of 3. Now if the transmission of $x^{t-2}_j$ experiences a delay of 2, then $\tau_{ji}(t) = 2$, since 
$x^{t-2}_j$ is the most recent estimate. See \Cref{fig: delay_illustration} for this illustration. The delay variables therefore measure the AoI with respect to the global clock at the receiving agents. 
\begin{figure}[!t]
	\centering
	\includegraphics[height=.2 \textheight]{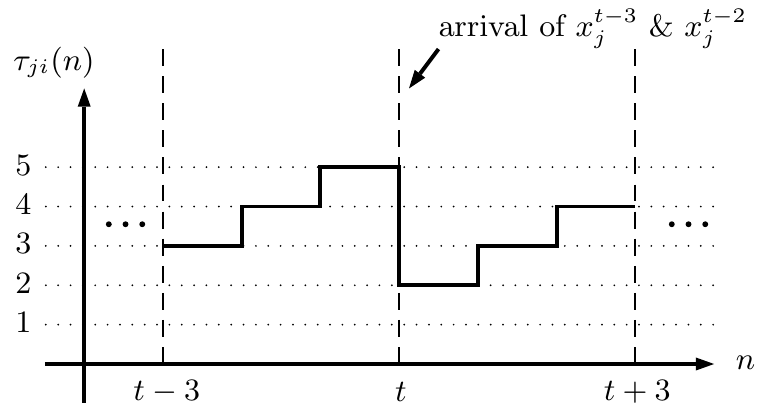}
	\caption{Illustration of the age of information variable $\tau_{ji}(n)$. We assume $\tau_{ji}(t-3) = 3$ and that no additional arrivals occur in $[t-3,t+3]$ except the two arrivals at time $t$.}
	\label{fig: delay_illustration}
\end{figure}
It is important to note that neither $n$ nor the $\tau_{ji}(n)$'s are explicitly used by the agents; these variables are solely for the sake of analysis.

The following SGD iteration is used by each agent in our multi-agent system to update its local variable:
\begin{multline}
\label{eq: real_iteration}
x^{n+1}_i = x^n _i - a(\nu(n, i)) \Ind{Y^n}(i) \\ \left( b(\nu(n, i))  \nabla_{x_i} f(x^{n - \tau_{1i}(n)} _1, \ldots, x^{n - \tau_{Di}(n)} _D, \xi^n)  + \nabla_{x_i} P(x^{n - \tau_{1i}(n)} _1, \ldots, x^{n - \tau_{Di}(n)} _D) + \varepsilon^n _i \right),
\end{multline}
where 
\begin{enumerate}
	\item $\{a(n)\}_{n \ge 0}$ is the given step-size sequence.
	\item $\{b(n)\}_{n \ge 0}$ is the given penalty-parameter sequence.
	\item $Y^n$ denotes the (possibly empty) set of all agents that update their variable during time slot $n$. $Y^n$ tracks the asynchronicity of the updates. Say that the agent~$i$ clock did not tick between $n-1$ and $n$ (no update of the local variable), then $i \notin Y^n$, $\Ind{Y^n}(i)=0$ and $x^{n+1}_i = x^n _i$. Note that the only other situation in which this can occur is when both gradients in \eqref{eq: real_iteration} are exactly zero.
	\item $\varepsilon^n_i$ is a stochastic additive error term that may arise during calculations of $\nabla f$ and $\nabla P$.
\end{enumerate}

The SGD iteration can be interpreted as sequence of approximate solutions to problems of the form \eqref{eq: penalty_problem}, where the parameter $b$ is gradually decreased. The step size sequence for the gradient descent iteration and the penalty parameter sequence will be chosen such that the change in the objective function $b(n)f(x,\xi) + P(x)$ appears stationary from the perspective of the optimisation variable iteration $x^n$ \eqref{eq: real_iteration}. 

\Cref{alg: 1} combines the local gradient descent updates with the CP of the previous subsection. The main events considered by the algorithm are the local clock tick and the arrival of new data from other agents. Recall that each agent maintains a local belief $\hat{X}_i \coloneqq (\hat{x}_{1i}, \ldots, \hat{x}_{Di})$ of the global optimisation vector. Further, each agent~$i$ appends its estimate $\hat{x}_{ii} = x^n_i~\forall n\ge0$ with the local timestamp $\nu(i)$ to facilitate consistent updates. Then, whenever $\hat{X}_i$ is updated, each agent schedules $\hat{X}_i$ for a transmission to its neighbours. Clearly there are additional events, e.g. like the event that the network topology changes. Hence, \Cref{alg: 1} only shows the main components of the optimisation/communication protocol.

\begin{algorithm}[!t] 
	\SetAlgoLined
	Initialize local optimisation variable estimate $\hat{x}_i$ \; 
	Initialize local belief vector $\hat{X}_i$ \;
	Set $\nu(i) = 1$ \;
	\While{True}
	{
		ev = wait for event;
		
		\uIf{ev is local clock tick}{
			1. Obtain network sample $\xi$ \;
			2. $\hat{x}_i \leftarrow \hat{x}_i - a(\nu(i)) \left(b(\nu(i)) \nabla_{x_i} f(\hat{X}_i, \xi)  +  \nabla_{x_i} P(\hat{X}_i) \right)$ \;
			3. Update $\hat{x}_{ii}$ of $\hat{X}_i$ to the new $\hat{x}_i$ and append the local timestamp $\nu(i)$ \;
			4. $\nu(i) = \nu(i) + 1$ \; 
		}
		\ElseIf{ev is transmission received from any agent $j$}{
			4. Receive possibly new components of $\hat{X}_j$ \;
			5. Update $\hat{X}_i$ using successfully received data \;
		}
		\For{any neighbouring agent $j$}{
			6. Schedule transmission of new information in $\hat{X}_i$ to agent $j$ \;
		}
	}
	\caption{Local algorithm at agent $i$}
	\label{alg: 1}
\end{algorithm}

\subsection{Assumptions} 
\label{sec: asmp}

In the following we will present assumptions such that \Cref{alg: 1}, in combination with the network model in \Cref{sec: graph-based formulation}, \ref{sec: network_model} and \ref{sec: CP},  solves problem \eqref{eq: intro_problem}. Assumption \cref{asm: objective}--\cref{asm: add_error} are related to our gradient descent iteration. Assumption \cref{asm: network_specific} describes our network-related assumptions.

\begin{assumption}
	\label{asm: objective}
	\begin{enumerate}[ref={(A\theassumption)(\roman*)}]
		\item \label{asm: objective1} $\nabla_x f$ is continuous and locally Lipschitz-continuous in $x$, where the associated constant may depend on $\xi$. 
		\item \label{asm: objective2} $f$ satisfies conditions to change the order of expectation and differentiation.
		\item \label{asm: objective3} $\xi$ is an $\mathcal{S}$-valued random variable, where $\mathcal{S}$ is a compact space.
	\end{enumerate}
\end{assumption}

\begin{assumption}
	\label{asm: penalty/constraint}
	\begin{enumerate}[ref={(A\theassumption)(\roman*)}]
		\item \label{asm: penalty/constraint1} $\nabla_x P$ is locally Lipschitz-continuous.
		\item \label{asm: penalty/constraint2} $\nabla_x P(z) \not= 0$ for all $z\in \cX^c$.
		\item \label{asm: penalty/constraint3} For all $x \in \partial \mathcal{X}$ we have $\nabla P(y)^T(x-y)\le 0$ for all $y$ in a small neighbourhood of $x$.
	\end{enumerate}
\end{assumption}

\begin{assumption}
	\label{asm: step_size and penalty seq.}
	The step-size sequences $\{a(n)\}_{n \ge 0}$ and the penalty parameter sequence $\{b(n)\}_{n \ge 0}$ satisfy the following conditions:
	\begin{enumerate}[ref={(A\theassumption)(\roman*)}]
		\item \label{asm: step_size alpha_beta} $\sum \limits_{n \ge 0} a(n) = \infty$, $\sum \limits_{n \ge 0} a(n) ^2 < \infty$,
		$\sum \limits_{n \ge 0} b(n) = \infty$.
		\item \label{asm: step_size asymp} $\frac{b(n)}{a(n)} \rightarrow 0$ and $b(n)$ is nonincreasing.
		\item \label{asm: step_size decay} $\limsup \limits_{n \to \infty} \sup \limits_{y \in [x, 1]} \frac{a( \lfloor yn \rfloor)}{a(n)} < \infty$ for $0 < x \le 1$.
		\item \label{asm: step_size bounded} $\sup \limits_{n \ge 0} a(n) \le 1$.
		\item \label{asm: step_size ratio} For $m \le n$, we have $a(n) \le \kappa a(m)$, where $\kappa > 0$.
	\end{enumerate}
\end{assumption}
\begin{assumption}
	\label{asm: stability}
	$\sup \limits_{n \ge 0} \ \lVert x^n \rVert < \infty$ a.s.
\end{assumption}
\begin{assumption}
	\label{asm: async}
	Almost surely, $\liminf \limits_{n \to \infty} \frac{\nu(n, i)}{n} > 0$ for $1 \le i \le D$.
\end{assumption}
\begin{assumption}
	\label{asm: add_error}
	Almost surely, $\limsup \limits_{n \to \infty} \ \lVert \varepsilon^n \rVert \le \varepsilon$ for some fixed $\varepsilon > 0$.
\end{assumption}

For a detailed discussion of \cref{asm: objective}, \cref{asm: stability} and \cref{asm: add_error}, we refer the reader to our previous work \cite{arxiv}. \ref{asm: penalty/constraint1} is the analogon of \ref{asm: objective1} for the penalty function. \ref{asm: penalty/constraint2} states that the penalty function has no stationary points outside of the constraint set $\cX$. 
This assumption is \emph{necessary} since otherwise the algorithm can get stuck in a local minimum outside the constraint set if $P$ has stationary points in $\cX^c$. \ref{asm: penalty/constraint3} requires that the penalty function descends towards the boundary at a point in a small neighbourhood of the boundary. Typically, penalty functions are strictly increasing in the distance to the constrained set.

\ref{asm: step_size asymp} states that $b(n)$ decays quicker than $a(n)$. In the language of stochastic approximation this says that the change in the gradient function runs on a slower time scale than the change of the iteration $x^n_i$. Therefore, the gradient function appears as ``quasi-static" from the perspective of the optimisation variable iteration $x^n_i$ \cite{borkar1997stochastic}.
Examples that satisfy the remaining requirements of \cref{asm: step_size and penalty seq.} as well as \cref{asm: async} are given in the numerical simulation Section~\ref{sec: num_sim}. 

For our network-related assumptions, recall that $\cG= \{G^n\}_{n\ge 0}$ denotes the sequence of network graphs under consideration. Moreover, we defined the event $A^n_{ij} \coloneqq \{\text{SINR}^{n}_{ij} \ge \beta\}$.
\begin{assumption}[Network assumption]
	\label{asm: network_specific}
	\begin{enumerate}[ref={(A\theassumption)(\roman*)}]
		\item \label{asm: network_specific_1} For every edge $(i,j) \in E$, there is a graph $G \in \cG$ and some $p \in (0,1)$, such that
		\begin{equation}
			\label{eq: network_condition}
			\Pr{A^n_{ij} \mid G} > p > 0.
		\end{equation}
		
		\item \label{asm: network_specific_2} 
		For every edge $(i,j) \in E$, there exists $q \in (0,1)$, such that for $n > m \ge 0$, we have
		\begin{equation}
		\label{eq: dependency_decay}
		\Big\vert \Pr{A_{ij}^n} - \Pr{A^n_{ij} \mid A^1_{ij},\ldots, A^m_{ij}   } \Big\vert   \le q^{n - m}.
		\end{equation}
	\end{enumerate}
\end{assumption}

\ref{asm: network_specific_1} guarantees that for each network edge there is at least one viable network topology such that communication is possible with some positive probability in the presence of channel interference. 
One can sharpen \ref{asm: network_specific_1} by only requiring \eqref{eq: network_condition} for those edges that are necessary to make $\cG$ stochastically strongly connected. 
Additionally, one can allow $p$ to vary with time and even allows $p$ to approach 0. However, the speed at which $p$ would approach 0 needs to be controlled. The extension to these cases merely requires bookkeeping and some basic analysis of convergent series.
In \Cref{sec: practical+ext} we will formulate practical conditions to satisfy \ref{asm: network_specific_1}. For now we work with this abstract formulation.

\ref{asm: network_specific_2} requires that the dependency between the transmissions for each \emph{individual network edge} decay at an exponential rate. This is our time-diversity assumption for the network. Block fading channels and Markovian fading channels with exponential decaying autocorrelation are examples that satisfy \ref{asm: network_specific_2}.

\section{Analysis of \Cref{alg: 1}}

As the first step we analyse the AoI of \Cref{alg: 1} under \Cref{asm: network_specific}. 

\subsection{AoI analysis of \Cref{alg: 1}}
\label{sec: informationdelay}
We show that the AoI variables $\tau_{ij}(n)$ are stochastically dominated\footnote{A non-negative integer-valued random variable $\tau$ is said to be stochastically dominated by a random variable $\overline{\tau}$ if $\Pr{\tau > m } < \Pr{\overline{\tau} > m}$ for all $m \ge 0$.} by a random variable with finite second moment. 
This is the central network-related property, since this implies that $\Pr{\tau_{ij}(n) > \sqrt{n} \ i.o.} = 0$ and will allow us to prove that the errors of the gradient iteration associated with the AoI vanish asymptomatically. We start the proof with two lemmas.

The first lemma establishes a local stochastic dominance property. We show the AoI associated with each edge is stochastically dominated by a non-negative random variable with finite second moment. For each edge $(i,j) \in E$, we denote by $\tau_{(i,j)}(n)$ the AoI at time $n$ associated with the edge $(i,j)$. Note the distinction between $\tau_{(i,j)}(n)$ and $\tau_{i,j}(n)$, i.e.\ $\tau_{(i,j)}(n)$ is associated with the specific edge, while $\tau_{i,j}(n)$ is associated with information from node $i$ to $j$ that might also arrive at $j$ via a different path then the direct edge $(i,j)$. From this it directly follows that $\tau_{i,j}(n)$ is stochastically dominated by $\tau_{(i,j)}(n)$.

\begin{lemma}
	\label{lem: single_edge_dom}
	Assume \cref{asm: network_specific}, then for every edge $(i,j) \in E$, there exists a non-negative integer-valued random value $\overline{\tau}_{(i,j)}$ that stochastically dominates $\tau_{(i,j)}(n)$ for all $n \ge 0$ and $\Ew{\overline{\tau}_{(i,j)}^2} < \infty$.
\end{lemma}
\begin{proof}
	For a simpler presentation, we assume without loss of generality that the considered Bandwidth $B$ in \eqref{eq: bitrate} is large enough to allow an effective bitrate $R$ such that $A^n_{ij}$ allows a successful transmission of \emph{all waiting transmissions for agent $j$ from agent $i$} during time slot $n$.
	This is justified as follows. Our communication protocol implies that there are always finitely many (at most $D$) variables waiting for transmission from agent $i$ to agent $j$. Using the effective bit rate $R$ associated with the event $A^n_{ij}$, it is immediate that there is a fixed number $N$, such that if the event $A^{n+t}_{ij}$ occurs for at least $N$ time-steps then at least $x^n_i$ or a more recent version is available at agent $j$. It is a simple exercise to extend the following proof using this observation.
	
	Now fix $(i,j) \in E$. As the first step we establish an upper bound for
	\begin{equation}
	\label{eq: lemma_specific_inequality}
	\Pr{\tau_{(i,j)}(n) > m} =  \Pr{\bigcap_{t=n-m}^{n}   (A^t_{ij})^c } 
	\end{equation}
	for all $m \ge 0$, where $(A^t_{ij})^c = \{\text{SINR}^{n}_{ij} < \beta\} $ denotes the complement of the event $A^t_{ij}$. To use the exponential decay property \ref{asm: network_specific_2}, we consider only those events in \eqref{eq: lemma_specific_inequality} that are sufficiently separated in time. We need to make sure that the separation increases with $m$, while at the same time the number of considered events increases with $m$. Choose increasing indices $t_i$ with $n-m = t_1 < \ldots < t_i < \ldots < n $ divided by $\lceil \sqrt{m} \rceil$ steps. We get $N$  indices with  $  \sqrt{m+1} -2 \le N = \lfloor \frac{m}{\lceil \sqrt{m}\rceil} \rfloor \le \sqrt{m} $. We have that
	\begin{align}
	\Pr{\bigcap_{t=n-m}^{n}   (A^t_{ij})^c } &\le \Pr{\bigcap_{k=1}^{N} (A^{t_k}_{ij})^c}, \\
	&= \Pr{\bigcap_{k=1}^{N-1} (A^{t_k}_{ij})^c  } \Pr{(A^{t_N}_{ij})^c \mid \bigcap_{k=1}^{N-1} (A^{t_k}_{ij})^c  }, \\
	&\le \Pr{\bigcap_{k=1}^{N-1} (A^{t_k}_{ij})^c  }   \left( \Pr{(A^{t_N}_{ij})^c} + q^{\sqrt{m}}    \right), \\
	&\vdots \\
	&\le \prod_{k=1}^{N}\Pr{(A^{t_k}_{ij})^c} + N q^{\sqrt{m}}.
	\end{align}
	The first step uses monotonicity, the second step rewrites the joint event as product of the marginal probability and the conditional probability. Then step 3 applies \ref{asm: network_specific_2} with the associated $q$ for $(i,j)$. After that we repeat step 2 and 3. 
	
	It now follows from \ref{asm: network_specific_1}, the stochastic strong connectivity of $\cG$ and the law of total probability that
	\begin{equation}
	\label{eq: p_bound}
	\Pr{ (A^n_{ij})^c} < \tilde{p} < 1
	\end{equation}
	for some $\tilde{p} \in (0,1)$ independent of $n$.
	Hence, 
	\begin{align}
	\label{eq: lemma2_stochdom}
	\Pr{\tau_{(i,j)}(n) > m} \le \tilde{p}^{\sqrt{m+1}-2} + \sqrt{m}q^{\sqrt{m}}. 
	\end{align}
	
	Now, choose $M$ such that $\tilde{p}^{\sqrt{m+1}-2} + \sqrt{m}q^{\sqrt{m}} \le 1$ for $ m \ge M$.
	Let us now define a random variable $\overline{\tau}_{(i,j)}$ by describing its complementary cumulative distribution function as follows:
	\begin{align}
	\Pr{\overline{\tau}_{(i,j)} > m} &= 1, \quad \text{ for all } 0\le  m < M, \\
	\Pr{\overline{\tau}_{(i,j)} > m} &= \tilde{p}^{\sqrt{m+1}-2} + \sqrt{m}q^{\sqrt{m}}, \quad \text{ otherwise}. 
	\end{align}
	By equation \eqref{eq: lemma2_stochdom} it follows that $\tau_{(i,j)}(n)$ is stochastically dominated by $\overline{\tau}_{(i,j)}$ for all $n \ge 0$. Finally, it is easy to check that
	\begin{equation}
	\begin{split}
	\Ew{\overline{\tau}_{(i,j)}^2} &= \sum_{m=0}^{\infty} 2m \Pr{\overline{\tau}_{(i,j)} > m} \\
	&\le 2 \left( \sum_{m=0}^{\infty} m \tilde{p}^{\sqrt{m+1}-2} + m\sqrt{m}q^{\sqrt{m}} \right) < \infty.
	\end{split}
	\end{equation}
\end{proof}

The next elementary lemma relates the AoI variables of a network path $\cP$ with the AoI variables of a decomposition of $\cP$.
\begin{definition}
	Let $\cP$ be a network path. We say $(\cP_1,\cP_2)$ is a \textbf{decomposition} of $\cP$ if the concatenation of the paths $\cP_1$ and $\cP_2$ is equal to $\cP$.
\end{definition}
For a network path $\cP$, we denote by $\tau_\cP(n)$ the AoI associated with $\cP$ at time $n$. At time $n$, $\tau_\cP(n)$ is the AoI of information from the first node of $\cP$ at the last node of $\cP$ at time $n$.  
\begin{lemma}[Keylemma]
	\label{lem: keylemma}
	Let $(\cP_1,\cP_2)$ be a decomposition of a finite path $\cP$ on a graph $G$. Suppose that for all $n\ge 0$ the AoI $\tau_{\cP_1}(n)$ and $\tau_{\cP_2}(n)$ are stochastically dominated by non-negative integer-valued random variables $\overline{\tau}_{\cP_1}$ and $\overline{\tau}_{\cP_2}$, respectively.
	Then there exists a non-negative integer-valued random variable $\overline{\tau}_\cP$ that stochastically dominates $\tau_{\cP}(n)$ for all $n\ge0$ with
	\begin{align}
	\Ew{\overline{\tau}_\cP^2}
	\le 2 \left(\Ew{\overline{\tau}_{\cP_1}^2} + \Ew{\overline{\tau}_{\cP_2}^2} \right).
	\end{align}
\end{lemma}
\begin{proof}
	Fix $m\ge 2$ and observe that 
	\begin{equation}
	\label{eq: keylemma_property}
	\{\tau_{\cP_1}(n-\nicefrac{m}{2}) \le \nicefrac{m}{2}\} \cap \{\tau_{\cP_2}(n) \le \nicefrac{m}{2}\} \subset  \{\tau_{\cP}(n) \le m\}, 
	\end{equation}
	i.e.\ an AoI of less than $\nicefrac{m}{2}$ for information received via $\tau_{\cP_1}$ at time $n-\nicefrac{m}{2}$ and an AoI of less than $\nicefrac{m}{2}$ for information received via $\tau_{\cP_2}$ at time $n$ imply an AoI of less than $m$ for information received via $\tau_{\cP}$ at time $n$.
	By taking the complement of \eqref{eq: keylemma_property}, this implies 
	\begin{align}
	\Pr{\tau_{\cP}(n) > m} &\le  \Pr{ \{\tau_{\cP_1}(n-\nicefrac{m}{2}) > \nicefrac{m}{2}\} \cup \{\tau_{\cP_2}(n) > \nicefrac{m}{2}\} } \\
	&\le  \Pr{\tau_{\cP_1}(n-\nicefrac{m}{2}) > \nicefrac{m}{2} }  + \Pr{\tau_{\cP_2}(n) > \nicefrac{m}{2}} \\
	&< \Pr{ \overline{\tau}_{\cP_1} > \nicefrac{m}{2}} + \Pr{\overline{\tau}_{\cP_2} > \nicefrac{m}{2}}.
	\end{align}
	In the last step, we use the assumption that there are random variables $\overline{\tau}_{\cP_1}$ and $\overline{\tau}_{\cP_2}$ that stochastically dominate $\tau_{\cP_1}(n)$ and $\tau_{\cP_2}(n)$, respectively, for all $n$. 
	
	Now $\overline{\tau}_{\cP_1}$ and $\overline{\tau}_{\cP_2}$ are integer-valued, so there is some $M \in \N$ such that $\Pr{\overline{\tau}_{\cP_1} > \nicefrac{m}{2} }  + \Pr{\overline{\tau}_{\cP_2} > \nicefrac{m}{2}} \le 1$ for all $m\ge M$. Define a random variable $\overline{\tau}_{\cP}$ by defining its complementary cumulative distribution function as follows:
	\begin{align}
	\Pr{\overline{\tau}_{\cP} > m} &\coloneqq 1, \quad \text{ for all } 0 \le m < M, \\
	\Pr{\overline{\tau}_{\cP} > m} &\coloneqq \Pr{\overline{\tau}_{\cP_1} > \nicefrac{m}{2} }  + \Pr{\overline{\tau}_{\cP_2} > \nicefrac{m}{2}}, \quad \text{ otherwise}. 
	\end{align}
	Then 
	\begin{align}
	\Ew{\overline{\tau}_{\cP}^2} &= \sum_{m=0}^\infty 2 m \Pr{\overline{\tau}_{\cP} > m} \\
	&\le \sum_{m=0}^\infty 2 m \Pr{\overline{\tau}_{\cP_1} > \nicefrac{m}{2} } + \sum_{m=0}^\infty 2 m \Pr{\overline{\tau}_{\cP_1} > \nicefrac{m}{2}} \\
	&\le 2 \left(\Ew{\overline{\tau}_{\cP_1}^2} + \Ew{\overline{\tau}_{\cP_2}^2} \right)
	\end{align}
\end{proof}

\subsection{Main results}
\label{sec: conv_results}
\begin{theorem}
	\label{thm: information_age}
	Under \cref{asm: network_specific}, there exists a non-negative integer-valued random variable $\overline{\tau}$ that stochastically dominates all $\tau_{ij}(n)$ with $\Ew{\overline{\tau}^2} < \infty$ and  $\Pr{\tau_{ij}(n) > \sqrt{n} \ i.o.} = 0$.
\end{theorem}
\begin{proof}
	Recall that $G=(V,E)$ is the union graph associated with the time-varying network graph $\cG$. Since $\cG$ is stochastically strongly connected, we can construct a finite network path $\cP = \left((i_k,i_{k+1})\right)_{k=1}^{m-1}$ of edges in $E$ for some $m\in \N$, where \emph{every node of $V$ is traversed at least twice}. 
	Therefore, we have that all AoI variables $\tau_{ij}(n)$ are stochastically dominated by $\tau_{\cP}(n)$, the AoI associated with $\cP$.
	\Cref{lem: single_edge_dom} shows that for each edge $(i_k,i_{k+1})$ there is a non-negative integer valued random variable $\overline{\tau}_{(i_k,i_{k+1})}$ that stochastically dominates $\tau_{(i_k,i_{k+1})}(n)$ for all $n \ge 0$ with $\Ew{\overline{\tau}_{(i_k,i_{k+1})}^2} < \infty$.
	Successive application of \Cref{lem: keylemma} now shows that there exists non-negative integer valued random variable $\overline{\tau}$ that stochastically dominates $\tau_{\cP}(n)$ for all $n\ge0$, with
	\begin{align}
	\label{eq: path_inequality}
	\Ew{\overline{\tau}^2} \le \sum_{k=1}^{m-1} 2^{m-k} \Ew{\overline{\tau}_{(i_k,i_{k+1})}^2} < \infty.
	\end{align}
	
	Finally, we have that for all $i,j \in V$
	\begin{align}
	\sum_{n=0}^{\infty} \Pr{\tau_{ij}(n) > \sqrt{n}} &\le \sum_{n=0}^{\infty} \Pr{\tau_\cP(n) > \sqrt{n}} \\
	&\le \sum_{n=0}^{\infty} \Pr{\overline{\tau} > \sqrt{n}} = \sum_{n=0}^{\infty} \Pr{\overline{\tau}^2 > n} = \Ew{\overline{\tau}^2} < \infty.
	\end{align}
	The Borel-Cantelli lemma therefore shows that $\Pr{\tau_{ij}(n) > \sqrt{n} \ i.o.} = 0$.
\end{proof}

We proved that there exists a sample path dependent $N \in \N$, such that $\tau_{ij}(n) \le \sqrt{n}$ for all $n\ge N$. Using \Cref{thm: information_age} and the assumptions associated with our gradient descend iteration, we can formulate the following convergence result, which we prove in Appendix \ref{sec: conv_analysis}. 

\begin{theorem}\label{thm: conv_analysis}
	Let \cref{asm: objective}--\cref{asm: network_specific} be satisfied, then \Cref{alg: 1} converges, almost surely, to a small neighbourhood of the set of local minima $\cX^*$ of problem \eqref{eq: intro_problem}
	\begin{equation}
	\lim\limits_{n \rightarrow \infty} x^n \in \overline{B}_{D\varepsilon}( \cX^*)  \quad a.s.,
	\end{equation}
\end{theorem}
where $D$ is the number of agents and $\varepsilon$ is the considered upper bound for the stochastic additive error term that may arise during calculations of $\nabla f$ and $\nabla P$. 

\section{Practical verifiability}
\label{sec: practical+ext}

In the previous sections we worked with abstract conditions for the event $A^n_{ij} = \{\text{SINR}^{n}_{ij} \ge \beta \}$. In this section we discuss how these conditions can be practically satisfied.

\subsection{Additive SINR model}

In general, we view any communication as a broadcast transmission where the received signal strength at different receivers is described by a signal power model. Therefore, the transmissions between any two nodes in the network can lead to interference at the receiver of a transmission between any two other nodes. Clearly, the effect should decay as the transmissions get more separated in time, frequency, space or code.

Fix a time slot $n$ and consider the associated network graph $G^n = (V,E^n)$ as well as an edge $(i,j) \in E^n$. Consider a transmission from node $i$ to node $j$ during time slot $n$. We denote the signal attenuation of the signal from node $i$ at node $j$ by the random variable $\alpha^n_{ij}$. We assume that $\alpha^n_{ij}$ is constant during each time slot $n$ analogously to the assumption that $\text{SINR}^{n}_{ij}$ is constant during each time-slot in \Cref{sec: network_model}. 
\emph{With a slight abuse of notation, by  $i' \not= i$ we always refer to those other nodes that possibly transmit while $i$ is transmitting to $j$ during some time slot $n$}, i.e.\ $i' \not= i$ will be exactly those nodes that have an outgoing edge according to $E^n$, except for $i$ and $j$. In the following, $n$ will always be clear from the context. We then use the standard additive model for the instantaneous SINR of the transmission from node $i$ to $j$ during time slot~$n$
\begin{equation}
\label{eq: SINR_model_general}
\text{SINR}^{n}_{ij} = \frac{p^n_{ij}\alpha^n_{ij}}{ \sum_{i' \not=i} p^n_{i'j}\alpha^n_{i'j}  + N_0},
\end{equation}
with constant background noise $N_0$ and transmission powers $p^n_{ij}$. 
In a wireless scenario, the signal attenuation $\alpha^n_{ij}$ will typically be the squared fading gain of the channel between $i$ and $j$. See for example the standard multi-user uplink flat AWGN fading channel model~\cite{tse2005fundamentals}.


The following lemma provides a simple condition that can be used to satisfy \ref{asm: network_specific_1}. It is practically attractive, since in only requires first order statistical information.
\begin{lemma}
	\label{lem: single_edge_event}
	For every $G^n \in \cG$, if 
	\begin{equation}
	\label{eq: network_condition_1}
	\frac{\Ew{p^n_{ij} \alpha^n_{ij}}}{\sum_{i' \not= i }  \Ew{p^n_{i'j} \alpha^n_{i'j}} + N_0} > \beta
	\end{equation}
	then there is some $p \in (0,1)$, such that
	\begin{equation}
	\Pr{(A^n_{ij})^c \mid G^n} < p < 1. 
	\end{equation} 
\end{lemma}
\begin{proof}
	Fix $(i,j) \in E^n$ for some a graph $G^n \in \cG$. 
	For simplicity, in this proof we refer with $\alpha^n_{ij}$ to the product $p^n_{ij} \alpha^n_{ij}$. From the additive SINR model we have that
	\begin{align}
	\Pr{\text{SINR}^{n}_{ij} < \beta \mid G^n} 
	= \Pr{Y^n_{ij} + \beta N_0 > 0},
	\end{align} 
	with $Y^n_{ij} \coloneqq -\alpha^n_{ij} + \beta \sum_{i'\not= j} \alpha^n_{i'j}$.
	We bound $\Pr{Y^n_{ij} \beta N_0 >0}$ using Chernoffs inequality and the moment-generating function (MGF) of $Y^n_{ij}$.
	Since $Y^n_{ij}$ is a linear combination of possibly dependent random variables, we then use the Cauchy-Schwartz inequality. Hence, for every $t>0$ we have an upper bound for the MGF of $Y^n_{ij}$:
	\begin{equation}
	\begin{split}
	M_{Y^n_{ij}}(t) = \Ew{e^{t Y^n_{ij}}} \le \sqrt{M_{\alpha^n_{ij}}(-2t) \prod_{i' \not= i } M_{\alpha^n_{i'j}}(2\beta t) }
	\end{split}
	\end{equation}
	Therefore, we have
	\begin{equation}\label{eq: lemma_specific_inequality_3}
	\begin{split}
	\Pr{\text{SINR}^{n}_{ij} < \beta \mid G^n} &\le \inf_{t>0} e^{\beta N_0 t} \sqrt{M_{\alpha^n_{ij}}(-2t) \prod_{i' \not= i} M_{\alpha^n_{i'j}}(2\beta t) }, \\
	&= \inf_{t>0} \sqrt{e^{2\beta N_0 t}  M_{\alpha^n_{ij}}(-2t) \prod_{i' \not= i} M_{\alpha^n_{i'j}}(2\beta t) } \\
	&=: \inf_{t>0} \sqrt{g_{ij}(t)}
	\end{split}
	\end{equation}
	
	Now observe that MGFs are logarithmically convex. It therefore follows that $g_{ij}(t)$ is a product of log-convex functions and therefore convex, since the set of log-convex functions is closed under products\cite[Thm. F, pg. 19]{roberts1973_convex}. It therefore follows that
	\begin{equation}
	\Pr{\text{SINR}^{n}_{ij} > \beta \mid G^n} \le  \inf_{t>0} \sqrt{g_{ij}(t)} < 1
	\end{equation}
	if
	\begin{equation}
	\label{eq: condition for A8_2}
	\frac{\mathrm{d}}{\mathrm{d}t} g_{ij}(t) \Big\vert_{t=0} = 2\beta N_0 -  2 \Ew{\alpha^n_{ij}} + 2 \beta \sum_{i' \not= i} \Ew{\alpha^n_{i'j}}  < 0.
	\end{equation}
\end{proof}

\subsection{Practical conditions}

Let $\cG$ be the sequence of directed network graphs. For simplicity assume that every edge in the union graph is required to make $\cG$ stochastically strongly connected. Hence, we need to satisfy \ref{asm: network_specific_1} for every edge.
Condition \eqref{eq: network_condition_1} provides a starting point to formulate networking protocols such that their combination with \Cref{alg: 1} can solve distributed constrained optimisation problems. Notice that \eqref{eq: network_condition_1} requires information that can be estimated locally at agent $j$ for agent $i$. Specifically, agent $j$ can estimate the mean interference noise term $\Ew{\text{IN}^n_{ij}} \coloneqq \sum_{i' \not= i}  \Ew{p^n_{ij} \alpha^n_{i'j}} + N_0$ as well as the signal attenuation mean $\Ew{\alpha^n_{ij}}$. Then agent $i$ can adapt its power policy using this information. Of course this comes at the cost of additional communication. For this reason, the design of a distributed protocol to satisfy the required, highly conflicting conditions in \ref{asm: network_specific_1} seams inappropriate.
	
Our perspective is that is desirable that every agent can learn a \emph{local power scheduling protocol} to optimise its AoI variables, while at the same its local variable should arrive at its neighbours frequently (Note that an agent that transmits less often does cause less interference for the other agents and we therefore expect that its experienced AoI is lower).
This opens the question, whether the agents can learn how to communicate to solve the distributed optimisation problem. It is very unclear whether an online multi agent learning algorithm for the design of the local power scheduling protocols can guarantee the conditions of \cref{asm: network_specific} during runtime. However, we can use \Cref{lem: single_edge_event} to design an epsilon greedy policy that may be used with any multi agent learning algorithm.  
	
Suppose every agent runs a local power scheduling protocol $p^n_{ij}$. For simplicity, assume a stationary communication environment, e.g. fixed fading distribution. Then it is reasonable to assume that after some time each agent can obtain local estimates of the signal attenuation mean $\Ew{\alpha^n_{ij}}$ as well as the interference noise term $\Ew{\text{IN}^n_{ij}}$ for each of its neighbours and for the observed network graphs in $\cG$. For all agents, fix some small $\varepsilon \in (0,1)$. 
Then for every edge $(i,j) \in E$. Define the $\varepsilon$-greedy local power scheduling protocol as
\begin{equation}
\label{eq: epsilon_greedy}
p^n_{\varepsilon, ij} \coloneqq \begin{cases}
\frac{\beta \Ew{\text{IN}^n_{ij}}}{\Ew{\alpha^n_{ij}}} + \delta, &\qquad \text{with probability } \varepsilon, \\
p^n_{ij} &\qquad \text{with probability } 1-\varepsilon.
\end{cases}
\end{equation}
with some small $\delta >0$. It is now the direct consequence of \Cref{lem: single_edge_event} that \ref{asm: network_specific_1} holds.
 
The next natural step is to analyse the above $\varepsilon$-greedy policy with (i) an online multi agent learning algorithm for the local power scheduling protocol and (ii) a non-stationary communication environment. Intuitively, it must be guaranteed that the learning algorithm and the environment run on a \emph{slower time-scale} than the estimators for the non-stationary environment. Then we expect that the local estimators can track the first order informations up to some accuracy and therefore we can adapt $\delta$ in \eqref{eq: epsilon_greedy} during runtime such that the $\varepsilon$-greedy policy can guarantee \ref{asm: network_specific_1}.
	 
\section{Numerical example}
\label{sec: num_sim}

This section presents a numerical example that illustrates the convergence of \Cref{alg: 1} in the extreme scenario where the average AoI grows with time.

\subsection{Sensor networks coverage problem}
\label{sec: exp_sensors}

For our numerical example, we consider 16 mobile agents in $\R^2$ that seek to solve a stochastic sensor coverage problem \cite{Wang2017coverage}. Each agent is equipped with a sensing unit: A target at position $y \in \R^2$ can be successfully detected by agent $i$ at $x_i \in \R^2$ with probability
\begin{equation}
	p(x_i,y, \xi) = \exp(-\xi \norm{x_i-y}^2),
\end{equation}
where $\xi$ is a positive random variable that takes values in a compact interval. The random variable $\xi$ can represent environmental or sensor related uncertainty. A single agent may not be able to sufficiently cover a region in $\R^2$. Collaboratively, the agents can minimise the probability that all sensors fail to successfully sense a position averaged over all positions in a region of interest. We consider this as the objective of this problem. To minimise the objective each agent has to specify its position. For the region of interest, we consider the unit disk $\overline{B}_1(0)$. We assume that the detection success of the agents are independent. The detection error probability at $y\in \R^2$ is therefore $p_e(x,y,\xi) = \prod_{i=1}^{D} (1- p(x_i,y, \xi))$. 
Additionally, we consider a finite set of sensor targets  $\cY \subset U$, with the requirement that given some $\delta \in (0,1)$ we have that almost surely
\begin{equation}
\label{eq: exp_constraints}
p_e(x,y, \xi) \le \delta
\end{equation}
for all $y \in \cY$. These targets are the constraints of this problem. We then use the Courant-Beltrami function \cite{chong2013introduction}
\begin{equation}
P(x, \xi) = \sum_{y\in \cY} (\max\{0, p_e(x,l,\xi)  - \delta\})^2
\end{equation} 
associated with the constraints as the penalty function. Formally, the objective is to minimise
\begin{equation}
\label{eq: sensor_problem}
F(x) = \E_\xi  \left[  \frac{1}{\pi} \int_{\overline{B}_1(0)} p_e(x,y, \xi) dy \right],
\end{equation}
subject to  $\E_\xi [P(x) ] = 0$. Note that $\E_\xi [P(x) ] = 0$ implies \eqref{eq: exp_constraints} almost surely. Also note that it is a trivial extension to consider that the penalty function also depends on the random variable $\xi$.
We consider that the system of agents has to solve this problem in a distributed manner by exchanging information over a communication network. Since the agents can only observe samples $\xi$ of the sensing uncertainty, they have to iteratively refine their positions. For simplicity we consider no agent dynamics, i.e. change of position is instantaneous. 

\subsection{Communication network}

For the communication network we use 8 ergodic Markov fading channels \cite{wang1995finitestate}. \emph{Our theory allows the fading channels to be correlated}. We emulate this by conditioning the transition dynamics of each fading channel on an additional Markov chain. One may interpret this as a partially observable system. 
In a standard Markov fading channel model, communication is successful in some states $s$ with some fixed probability $p_s >0$, i.e. to each state we associate a Bernoulli processes. Then, \ref{asm: network_specific_1} would immediately follow from the ergodicity of the channels, assuming the $8$ channels are scheduled to guarantee a stochastically strongly connected graph. Here, however, we consider an \emph{artificial, extreme scenario}. 
In \Cref{sec: asmp}, we remarked that \ref{asm: network_specific_1} is merely to simplify the presentation and that one can allow that the success probabilities of the network edges decay to zero. Here, we exemplify this and set the probability of a successful transmission to $p^n_s = c_s(1-\exp(-\frac{\sqrt{n}}{n}))$ for different $c_s \in (0,1)$ for all channel states in all 8 Markov channels. This induces that \emph{the AoI experienced by the agents will grow over time}. Our numerical example will show that \Cref{alg: 1} can still solve the described problem. 

\subsection{Verification of \cref{asm: objective} - \cref{asm: network_specific}}

First observe that $p_e(x,y,\xi)$ is smooth. Then \ref{asm: objective1}, \ref{asm: penalty/constraint1} and \ref{asm: objective2} will follow as a consequence of \cref{asm: stability}. The space $\cS$ is a compact interval so \ref{asm: objective3} holds. Now observe that $p_e(x,y,\xi)$ is increasing as a function of the distance $\norm{x_i -y}$ for all agents $i$ and all targets $y$. Hence, \ref{asm: penalty/constraint2} and \ref{asm: penalty/constraint3} follow.
For \Cref{alg: 1} we use step-size sequence $a(\nu(n,i)) = \frac{1}{\frac{\nu(n,i)}{1000} + 10}$ and the penalty parameter sequence $b(\nu(n,i)) = \frac{1}{\frac{\sqrt[3]{\nu(n,i)^2}}{1000} + 10}$ for all $n\ge 0$ and $1 \le i \le 16$. These sequences satisfy \cref{asm: step_size and penalty seq.}. 
We emulate asynchronous behaviour by letting each agent $i$ have a local clock that ticks at every increment of a homogeneous Poisson point process on the positive half-line with some positive rate. If follows from the strong law of large numbers for homogeneous Poisson point processes that \cref{asm: async} holds \cite{poisson2017lectures}. 
Assumption \cref{asm: stability} can often be difficult to verify. However, most of the time one can associate a projective scheme to the algorithm iterates. We refer to \cite[Sec. 5]{ramaswamy2020deepMAS} for details.

For the network assumptions, it follows from the exponential decay of ergodic Markov chains to their stationary distributions that \ref{asm: network_specific_2} is satisfied. In the previous subsection we discussed that we go beyond \ref{asm: network_specific_1}. Then \Cref{thm: information_age} is satisfied in the setup of the previous subsection, if it is ensured that the set of time-varying network topologies is stochastically strongly connected. To ensure this, we schedule the 8 communication channels independently and uniformly at random to a set of directed network topologies $\{(V, E_i)\}_{i \in I}$, such that $(V, \bigcup_{i\in I} E_i)$ is strongly connected. Here, this is possible with the 4 directed graphs shown in \Cref{fig: network_topologies}. We conclude that \Cref{thm: conv_analysis} holds and that \Cref{alg: 1} will find a local minimum of problem \eqref{eq: sensor_problem}.

\subsection{Evaluation}

We generate initial positions for all $16$ agents randomly outside the unit disk and set $\delta = \nicefrac{1}{4}$ for 5 targets on the unit disk. We let \Cref{alg: 1} run for $10000$ steps.
In \Cref{fig: aoi}, the first plot shows the average AoI experienced by the agents averaged over the delay associated with all components in all local belief vectors $\hat{X}_i$ at every time step. We see that the average AoI grows over time. The second plot illustrates that the average AoI is in $\cO(\sqrt{n})$.

\begin{figure}
	\centering
	\begin{minipage}{.48\textwidth}
		\centering
		\includegraphics[height=.25\textheight]{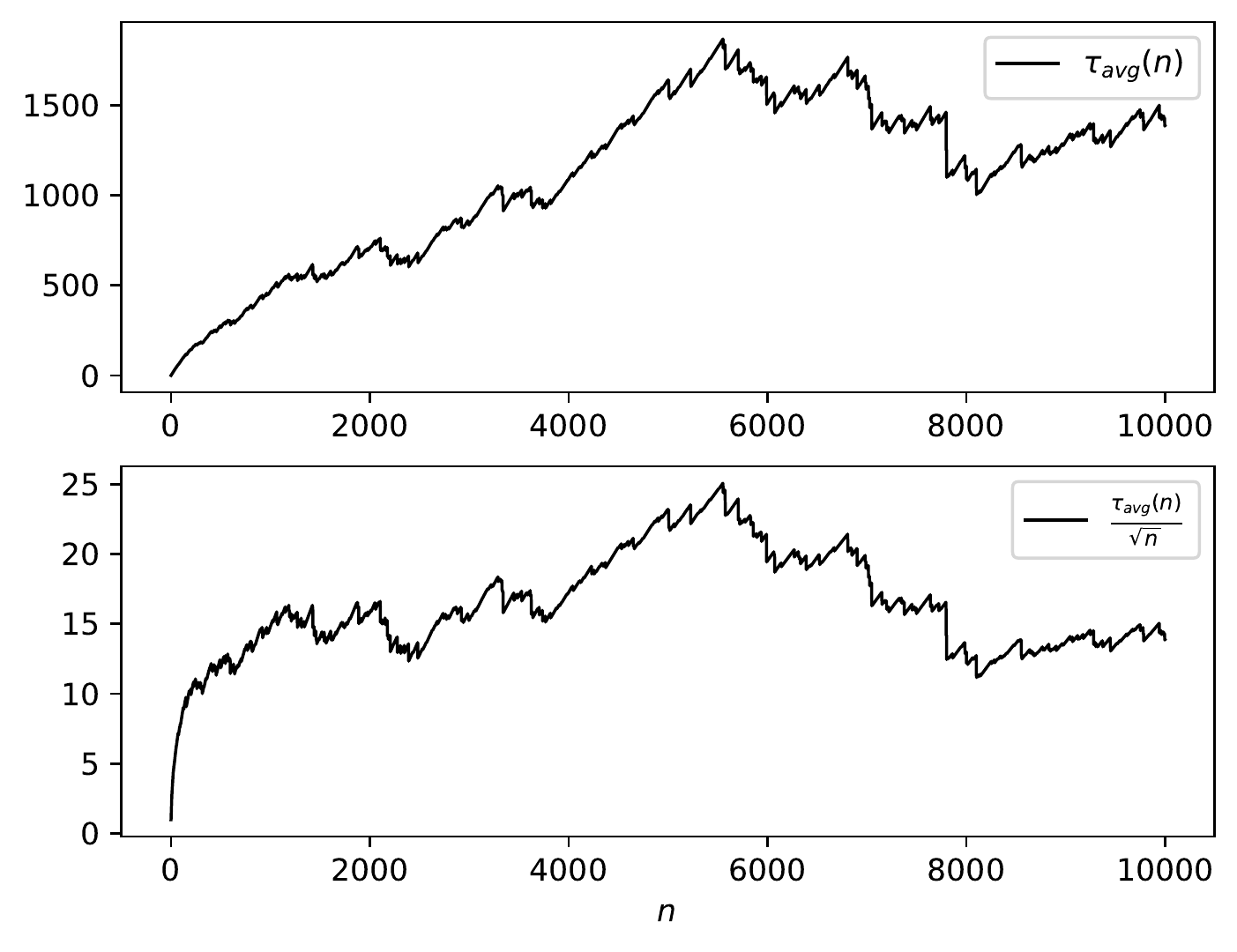}
		\captionof{figure}{Average AoI over all $\tau_{ij}(n)$} and average AoI normalised by $\sqrt{n}$. 
		\label{fig: aoi}
	\end{minipage}%
	\hfil
	\begin{minipage}{.48\textwidth}
		\centering
		\includegraphics[height=.25\textheight]{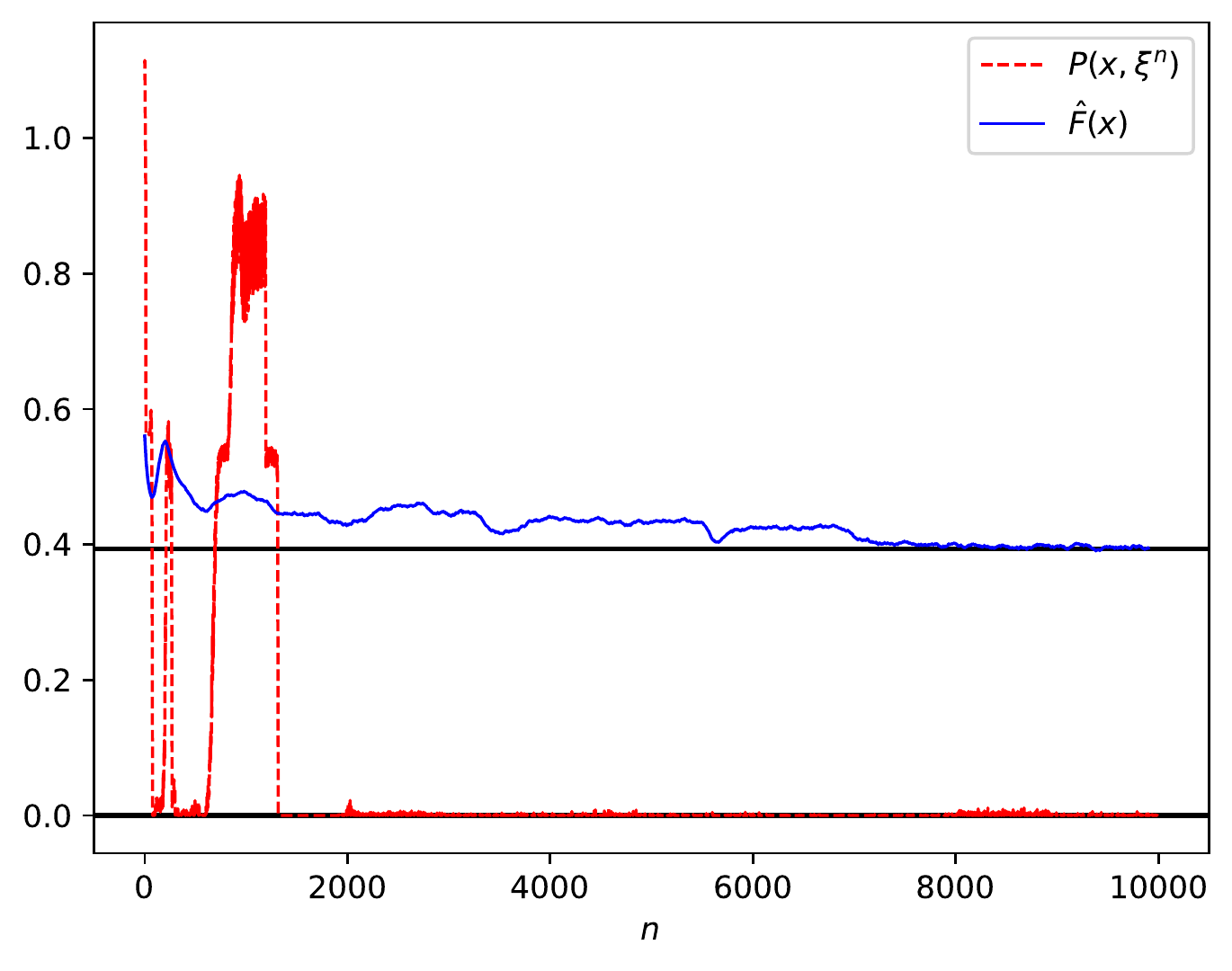}
		\captionof{figure}{Convergence of the objective and the penalty function values.}
		\label{fig: cost convergence}
	\end{minipage}
\end{figure}

\Cref{fig: cost convergence} shows the convergence of the average detection error probability. Notice that a major part of the convergence occurs after $n=1000$ and from \Cref{fig: aoi} we see that here the average AoI is already $\sim 500$ time steps and afterwards continuous to grow. We conclude that the algorithm convergence although the experienced average AoI is significant relative to the runtime of the experiment. In \Cref{fig: cost convergence}, we also plot the penalty function over time and, as a reference, the average detection probability for the unconstrained solution, which does not satisfy the constraints in \eqref{eq: exp_constraints}. We see that the algorithm solution satisfies the constraints, while the average detection quality approaches the solution of the unconstrained case.

\begin{figure}
	\centering
	\begin{minipage}{.48\textwidth}
	\centering
	\includegraphics[height=.25 \textheight, trim={0.5cm 0.5cm 0.5cm 0.25cm},clip ]{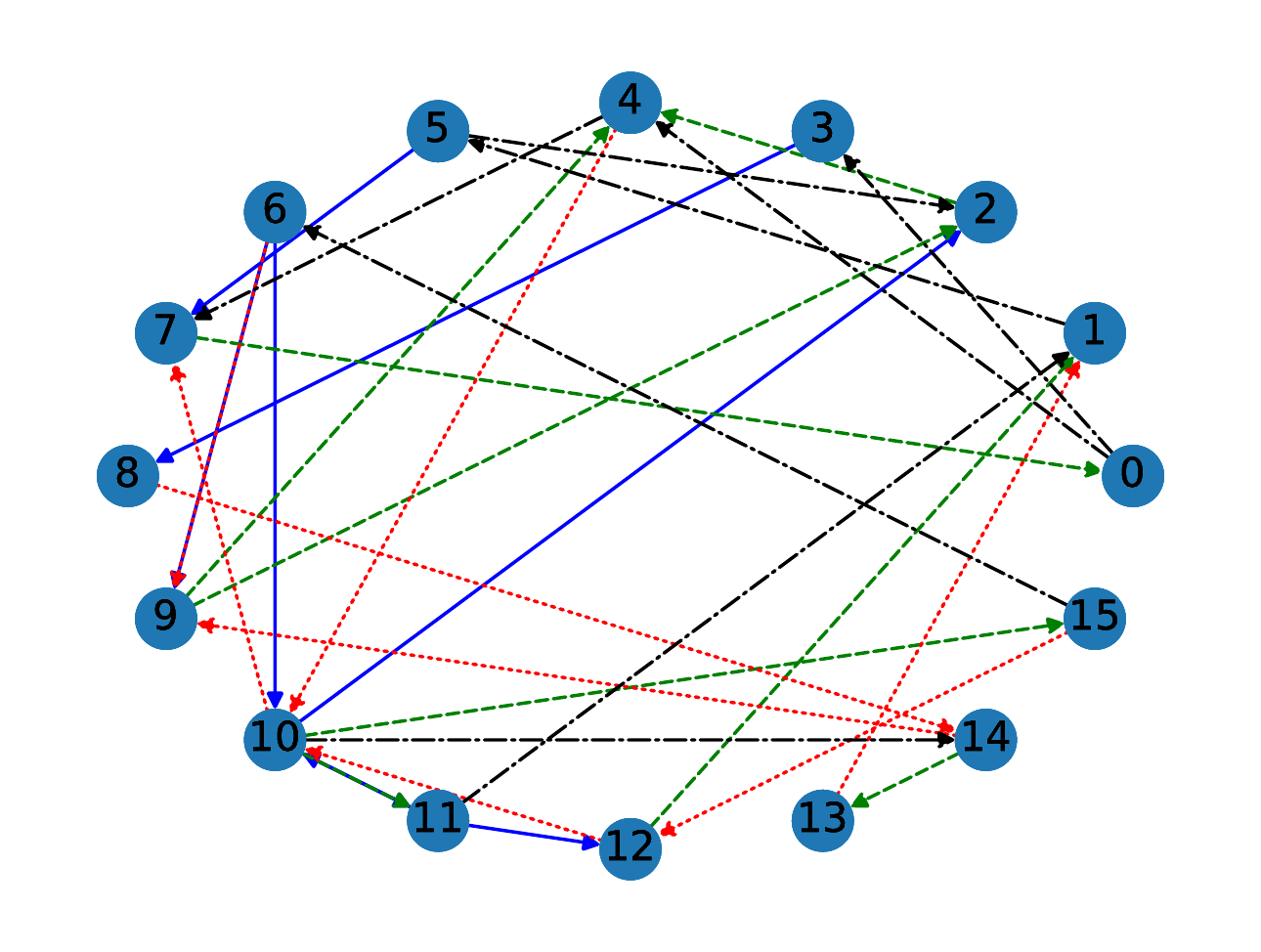}
	\caption{Four network topologies whose union is a strongly connected graph.}
	\label{fig: network_topologies}
	\end{minipage}%
	\hfil
	\begin{minipage}{.48\textwidth}
	\centering
	\includegraphics[height=.25\textheight]{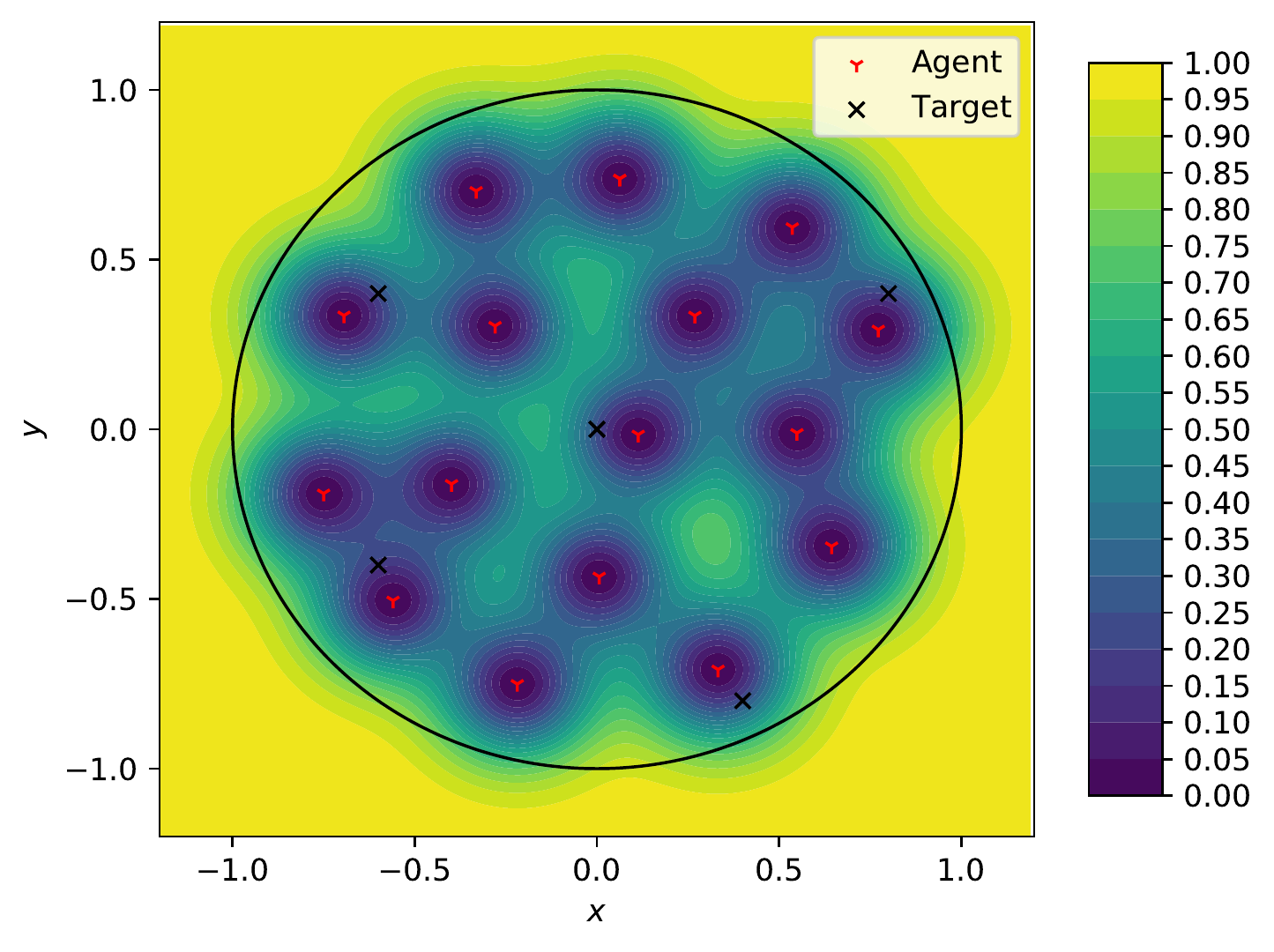}
	\caption{Final positions of agents in $\R^2$. The contour colours visualise the average error probability.}
	\label{fig: plane}
	\end{minipage}
\end{figure}

In \Cref{fig: plane}, we visualise the final positions of the agents and the resulting average detection probability in $\R^2$ as well as the positions of the targets. The asymmetry of the final agent positions is the consequence of the 5 constraints. In the unconstrained case, the final agent positions converge to symmetric arrangement of the agents.

\section{Conclusion and future work}

We presented an analysis of the age of information associated with data communicated over time-varying stochastically strongly connected networks and proved the convergence of a distributed optimisation algorithm under simple network assumption.
Our main conclusion is that assumption \cref{asm: network_specific} constitutes \emph{practically network conditions that can be verified locally}. We discussed that \ref{asm: network_specific_1} can be verified by an agent using first order statistically information communicated solely by neighbouring agents. Additionally, the exponential decay assumption \ref{asm: network_specific_2} is by definition a local property.

For future work, we would like to analyse our algorithm for constant step sizes and provide conditions to improve the rate of convergence.
Finally, as already mentioned in \Cref{sec: practical+ext}, \Cref{lem: single_edge_event} provides a starting point for the development of online network topology and transmission power scheduling protocols based on online estimations of the attenuation and interference-noise distributions. We will analyse these combinations in non-stationary environments using multi-timescale stochastic approximation.

\bibliographystyle{IEEEtran}
\bibliography{references}

\newpage
\begin{appendices}

\section{Proof of \Cref{thm: conv_analysis}}
\label{sec: conv_analysis}

For the analysis we combine the iterations in \eqref{eq: real_iteration} for all $i \in V$. With a slight abuse of notation, we use $\nabla_{x}F(x^{n - \tau_{1i}(n)} _1, \ldots, x^{n - \tau_{Di}(n)} _D, \xi^n)$ to denote the stacked gradient vector of the partial derivatives $\nabla_{x_i} F$ as a function of the optimisation variables $x^{n - \tau_{1i}(n)} _1, \ldots, x^{n - \tau_{Di}(n)} _D$ with the agent-specific delays $\tau_{ji}(n)$. We write the iteration in the following form:
\begin{multline}
\label{eq: iteration_analysis}
x^{n+1} = x^n -  \overline{a}(n) \lambda_a^n \Big[ \overline{b}(n) \lambda_b^n  \Big(\nabla_{x} F(x^{n - \tau_{1i}(n)} _1, \ldots, x^{n - \tau_{Di}(n)} _D, \xi^n)    + M^{n+1}  \Big) \\ + \nabla_{x} P(x^{n - \tau_{1i}(n)} _1, \ldots, x^{n - \tau_{Di}(n)} _D) + \varepsilon^n \Big],
\end{multline}
\\ where 
\begin{itemize}
	\item $\overline{a}(n) := \underset{i \in Y^n}{\max} \ a(\nu(n, i))$, 
	$q_a(n, i) := \frac{a(\nu(n, i)) \Ind{Y^n}(i)}{\overline{a}(n)}$
	\item $\overline{b}(n) := \underset{i \in Y^n}{\max} \ b(\nu(n, i))$, 
	$q_b(n, i) := \frac{b(\nu(n, i)) \Ind{Y^n}(i)}{\overline{b}(n)}$
	\item $\lambda_a^n := \text{diag}\left( \{Q_{1,a}^n\}_{i=1}^D \right),  Q^n_{i,a} = \text{diag}\left(\{q_a(n,i)\}_{k=1}^{d_i} \right)$
	\item $\lambda_b^n := \text{diag}\left( \{Q_{1,b}^n\}_{i=1}^D \right), Q^n_{i,b} = \text{diag}\left(\{q_b(n,i)\}_{k=1}^{d_i} \right)$
\end{itemize} 

Notice that the matrices $\lambda_a^n$ and $\lambda_b^n$ capture the
transient relative update and the relative objective function weighting of the agents with respect to $\overline{a}(n)$ and $\overline{b}(n)$, respectively. The agent updates are related to each other via these
maximum step-sizes. To analyse the discrete-time iteration \eqref{eq: iteration_analysis}, we will use $\{\overline{a}(n) \}_{n\ge0}$ to divide the time axis, and obtain a corresponding trajectory in the continuous domain. In this domain it will become clear how the change of the objective gradient will behave as stationary with respect to the iteration of the optimisation variable $x^n$. Further, the limiting behaviour of $\lambda_a^n$ and $\lambda_b^n$ will capture the asymptotic relative update frequency of the agents.

In the above new iteration \eqref{eq: iteration_analysis}, we have rewritten the sample gradient using the unknown gradient of $F$. First, consider the filtration generated from information up until time $n$:  $\mathcal{F}^0 \coloneqq  \sigma \langle x^0, \varepsilon^0, Y^0 \rangle$ and $ \mathcal{F}^n \coloneqq \sigma \langle x^m, \varepsilon^m, Y^m, \xi^{k} \mid \ m \le n, \ k < n \rangle$ for $n \ge 1$.
Then, we define $M_i ^{n+1} \coloneqq$ $\nabla_{x_i} f(x^{n - \tau_{1i}(n)} _1, \ldots, x^{n - \tau_{Di}(n)} _D, \xi^n)$   $-\mathbb{E} \left[ \nabla_{x_i} f(x^{n - \tau_{1i}(n)} _1, \ldots, x^{n - \tau_{Di}(n)} _D, \xi^n) \mid \mathcal{F}^n \right]$. Now, note that since the sequence $\{\xi^n\}_{n\ge0}$ is i.i.d. and $\mathcal{F}^n$ is generated from values up until $\xi^{n-1}$, we have that $\xi^n$ and $\mathcal{F}^n$ are independent. Hence, for all $n \ge 0$ and indices $i \in \mathcal{V}$, we have $\mathbb{E} \left[ \nabla_{x_i} f(x^{n - \tau_{1i}(n)} _1, \ldots, x^{n - \tau_{Di}(n)} _D, \xi^n) \mid \mathcal{F}^n \right] = \mathbb{E} \left[ \nabla_{x_i}f(x^{n - \tau_{1i}(n)} _1, \ldots, x^{n - \tau_{Di}(n)} _D, \xi^n) \right]$. Further, by \Cref{asm: objective}(ii), we can exchange the expectation and the gradient in the last statement, and arrive at the expression in \eqref{eq: iteration_analysis}.

First, we state with minor changes necessary lemmata from our previous work \cite{arxiv}. 
\begin{lemma}
	\label{lem: gl_lipschitz}
	Under \cref{asm: objective} and \cref{asm: penalty/constraint}, given a compact convex set $\mathcal{K} \subset \mathbb{R}^d$ and a compact set $\mathcal{S}$, there exists $0< L< \infty $ such that $\lVert \overline{b}(n) \lambda_b^n \nabla_x F(y, \xi) + \nabla_{x} P(y) - \left( \overline{b}(n) \lambda_b^n \nabla_x F(z, \xi) + \nabla_{x} P(z) \right) \rVert \le L \lVert y - z \rVert$ $\forall \ y,z \in \mathcal{K}$, $\xi \in \mathcal{S}$ and $n \ge 0$. 
\end{lemma}
\begin{proof}
	In \cite{arxiv} we proved that $\nabla_x F(x,\xi)$ is globally Lipschitz in the present setting. Now, as $\nabla_{x} P(x)$ is locally Lipschitz, $\nabla_{x} P(x)$ is also globally Lipschitz when restricted to the compact set $\mathcal{K}$. Finally, as $\norm{\overline{b}(n) \lambda_b^n} \le 1$ the Lemma follows.
\end{proof}
The first Lemma shows that when $ \Pi(n) \nabla_x F(x,\xi) + \nabla_{x} P(x)$ is restricted to a compact convex set, then it is globally Lipschitz with a Lipschitz constant independent of $\xi$ and $n$. Now, recall that $M^{n+1}$ captures the errors due to the use of samples instead of expected values. We define $\zeta^n := \sum \limits_{m=0}^{n-1} \overline{a}(m) \lambda^m_a \overline{b}(n) \lambda_b^n M^{m+1}$ to analyse the accumulated sampling error in \eqref{eq: real_iteration}. The following Lemma implies that the asymptotic behaviour of \eqref{eq: iteration_analysis} is identical to the version that uses expected values. The proof is identical to the one in \cite{arxiv} by replacing $\lambda^m_a$ with $\lambda^m_a \overline{b}(n) \lambda_b^n$.
\begin{lemma}\label{lem: martingale}
	Under \cref{asm: objective}--\cref{asm: async}, $\zeta^n$ is a zero-mean square integrable Martingale difference sequence and almost surely $\lim \limits_{n \to \infty} \zeta^n$ exists. 
\end{lemma}

We can now proof our convergence result. 
\begin{proof}[Proof of \Cref{thm: conv_analysis}]
	\emph{First part:} 
	The line of argument overlaps with the proof of \cite[Theorem 1]{arxiv} in our preceding paper, so we merely outline the main steps. It follows from \Cref{lem: martingale} that the accumulated error due to taking samples instead of expected values vanishes asymptotically.
	Now, the error due to communication delay incurred at time $n$ with respect to agent-$i$ is given by:
	\begin{multline}
	e^n_i \coloneqq a(\nu(n, i)) \Ind{Y^n}(i) \lVert b(\nu(n, i)) \Big(  \nabla_x F(x_1 ^n, \ldots, x_D ^n, \xi^n) -  \nabla_x F(x_1 ^{n - \tau_{1i}(n)}, \ldots, x_D ^{n - \tau_{Di}(n)}, \xi^n) \Big) \\+  \nabla_x P(x_1 ^n, \ldots, x_D ^n) -  \nabla_x P(x_1 ^{n - \tau_{1i}(n)}, \ldots, x_D ^{n - \tau_{Di}(n)}) \rVert.
	\end{multline}
	It follows from \cref{asm: stability} that almost surely $\{x^n\}$ can be restricted to a closed ball. Therefore, by Lemma \ref{lem: gl_lipschitz} we have
	\begin{equation} \label{eq: ana_1}
	e^n_i \le  a(\nu(n, i)) L \sum_{j \in V} \sum \limits_{m = n - \tau_{ji}(n)}^{n-1} \norm{x_j ^{m+1} - x_j ^m} .
	\end{equation}
	Now, using continuity of $\nabla_{x_j}F$ and $\nabla_{x_j}P$, \ref{asm: step_size bounded}, \ref{asm: step_size ratio}, \cref{asm: stability} and \cref{asm: add_error}, we get that almost surely
	\begin{equation}\label{eq: ana_2}
	\sum \limits_{m = n - \tau_{ji}(n)}^{n-1}\norm{x_j ^{m+1} - x_j ^m} \le \sum \limits_{m = n - \tau_{ji}(n)}^{n-1} a(m) C \le C \tau_{ji}(n)
	\end{equation}
	for some sample path dependent constant $C < \infty$. Finally, it follows from \Cref{thm: information_age} that there exists a sample path dependent $0 < N < \infty$ such that $\tau_{ij}(n) \le \sqrt{n}$ for all $n \ge N$. It now follows from $\sum \limits_{n \ge 0} a(n)^2 = \infty$ that $a(n) \in o(\nicefrac{1}{\sqrt{n}})$ and applying \cref{asm: async} therefore yields 
	\begin{equation}
	\label{eq: ana_3}
	a(\nu(n, i)) \tau_{ji}(n) \in o(1) \quad \forall n\ge N
	\end{equation} 
	Finally, combing \eqref{eq: ana_1}, \eqref{eq: ana_2} and \eqref{eq: ana_3}, we conclude that almost surely $e^n \in o(1)$ for $n\ge N$. 
	
	Under \cref{asm: objective}--\cref{asm: network_specific}, it follows from our analysis in \cite{arxiv} that for every fixed $\overline{b}\lambda_b$
	iteration $x^n$ asymptotically has the same set of limit points as the differential inclusion (DI)
	\begin{equation}
	\dot{x}(t) \in - \Lambda \bigl( \overline{b}\lambda_b \nabla_{x} F(x(t)) + \nabla_{x} P(x(t)) \bigr) + \overline{B}_{\varepsilon}(0),
	\end{equation}
	with $\Lambda \coloneqq \text{diag}(1/D, \ldots, 1/D)$ as a consequence of \cref{asm: async}, i.e. the use of balanced step-size sequences.
	The coincidence of the limits of the discrete gradient iteration and DI follows from associating continuous time trajectories to the sequences $x^n$ with respect to the step size sequence $\{ \overline{a}(n) \}$. 
	
	\emph{Second part:} 
	In this setting, it follows from $ \lim\limits_{n\rightarrow \infty} \frac{b(n)}{a(n)} = 0$, \ref{asm: step_size asymp} that $\overline{b}(t) \lambda_b(t)$ asymptotically appears stationary from the perspective of the trajectory of $x(t)$.
	Now, define the set valued map $\lambda: [0,1] \rightarrow \R^d$, where $\lambda(b)$ is the global attractor set of the differential inclusion $\dot{x}(t) \in - \left( b \Lambda \nabla_{x} F(x(t)) + \nabla_{x} P(x(t)) \right) $.
	From the above discussion we conclude that \eqref{eq: real_iteration} asymptotically tracks the set $\lambda(\overline{b}(n))$ with the same asymptotic error as $x(t)$, i.e.:
	\begin{equation}
	\label{eq: asm_track_short}
	d(x^n, \lambda(\overline{b}(n))) \rightarrow [0, \varepsilon D] \quad a.s.
	\end{equation}
	Here, one may also use a two timescale stochastic approximation view for the above setting. 

	\emph{Third part (Convergence to a local minimum):}
	First consider a convergent sequence $y^n \rightarrow y$ with $y^n \in \lambda(\overline{b}(n))$. Then
	\begin{equation}
	\label{eq: analysis_step3_condition}
	\overline{b}(n) \Lambda \nabla_{x} F(y^n) + \nabla_{x} P(y^n) = 0
	\end{equation}
	for all $n \ge 0$.
	Since $\lim\limits_{n\rightarrow \infty} \overline{b}(n) = 0$, we conclude by continuity of $\nabla_{x} P$ that $\nabla_{x} P(y) = 0$. By assumption \ref{asm: penalty/constraint2} we know that $\nabla_{x}P(y) \not= 0 $ for $y\in \cX^c$. Therefore $y \in \cX$. 
	
	We now show that the limit $y$ is a stationary point of problem \eqref{eq: intro_problem}. First, suppose $y$ is an interior point of $\cX$ with $\nabla_x F(y) \not=0$. This yields a contradiction since for some large $n$ we have that $y^n \in \cX$, while $F(y^n) \not=0$ by continuity. But $\nabla_{x}P(y^n) = 0$ so \eqref{eq: analysis_step3_condition} can not be satisfied.
	
	Second, suppose $y \in \partial \cX$.  If $\nabla_xF(y)^\top (z-y) < 0$ for some feasible direction with $z \in \cX$, then $y$ is no stationary point on the boundary.  W.l.o.g. we can consider two cases: First, suppose ${y^n}$ approaches $y$ from the interior of $\cX$. Then we arrive at a contradiction as above. So suppose ${y^n}$ approaches $y$ from $\cX^c$. Notice that we may choose $z$ arbitrarily close to $y$. Therefore by continuity and for some large $n$, we have
	\begin{equation}
	\nabla_xF(y^n)^\top (y-y^n) < 0.
	\end{equation}
	It now follows from \cref{asm: penalty/constraint3} that $\nabla_xP(y^n)^\top(y-y^n) < 0$ for some large $n$. Hence,
	\begin{equation}
	\overline{b}(n) \Lambda \nabla_x F(y^n)^\top (y-y^n) + \nabla_x P(y^n)^\top(y-y^n) < 0
	\end{equation}
	for some large $n$, which again contradicts \eqref{eq: analysis_step3_condition}  since $y^n \in \lambda(\overline{b}(n))$.
	
	Finally, observe that
	\begin{equation}
	y^n = \arg\min \{ \norm{y-x^n}_2~\mid y \in \lambda(\overline{b}(n)) \} \quad n \in \N,
	\end{equation} 
	defines a convergent sequence, which follows from the convergence of $x^n$ and since for all $\varepsilon>0$ we can choose some $N\in \N$ such that for all $n,m\ge N$ we have that $d( \lambda(\overline{b}(n)) , \lambda(\overline{b}(m)))~<~\varepsilon$.
\end{proof}

\end{appendices}
\end{document}